\documentclass[11pt]{amsart}

\usepackage{amsthm} 
\usepackage{amssymb}
\usepackage{graphicx}									 %Include graphics
\usepackage{stmaryrd}                  %Includes \mapsfrom
\usepackage{footmisc}                  %Something for footnotes
\usepackage{paralist}
\usepackage{wrapfig}
\usepackage[draft]{pdfcomment}
\usepackage{bbm}
\usepackage[ngerman,british]{babel}
\usepackage[arrow, matrix, curve]{xy}
\usepackage{color}
\usepackage{url}
\usepackage[pagewise]{lineno}%\linenumbers

\usepackage[T1]{fontenc}
\usepackage[utf8]{inputenc}

\newtheorem*{thm-plain}{Theorem}

\newtheorem{thm}{Theorem}[section]
\newtheorem*{thmw}{Theorem}
\newtheorem{lem}[thm]{Lemma}

\newtheorem{prp}[thm]{Proposition}

\theoremstyle{definition}
\newtheorem{dfn}[thm]{Definition}

\newtheorem*{question}{Question}

\theoremstyle{remark}

\newtheorem{rem}[thm]{Remark}

\newcommand{\R}{\mathbb{R}}

\newcommand{\N}{\mathbb{N}}
\newcommand{\Z}{\mathbb{Z}}
\newcommand{\Or}{\mathrm{O}}

\newcommand{\To}{\rightarrow}

\newcommand{\Orb}{\mathcal{O}}
	
\title[On the existence of closed geodesics on $2$-orbifolds]{On the existence of closed geodesics on $2$-orbifolds}
\author{Christian Lange}
\address{Christian Lange, Mathematisches Institut der Universit\"at zu K\"oln, Weyertal 86-90, 50931 K\"oln, Germany}
\email{clange@math.uni-koeln.de}
\thanks{The author is partially supported by the DFG funded project SFB TRR 191.}

\subjclass{53C22, 57R18}

\begin{document}

\begin{abstract} We show that on every compact Riemannian $2$-orbifold there exist infinitely many closed geodesics of positive length.
\end{abstract}\maketitle	

\section{Introduction}
\label{sec:Pse_Introduction}

Existence and properties of closed geodesics on Riemannian manifolds have been subject of intense research since Poincar\'e's work \cite{Poincare} from the beginning of the 20th century. A prominent result in the field is a theorem by Gromoll and Meyer that guarantees the existence of infinitely many closed geodesics on compact Riemannian manifolds on some cohomological assumption \cite{MR0264551}. This assumption is satisfied by a large class of manifolds but not by spheres. Generalizing ideas of Birkhoff \cite{MR0209095}, Franks \cite{MR1161099} and Bangert \cite{MR1209957} together proved the existence of infinitely many closed geodesics on every Riemannian $2$-sphere.

In \cite{MR2218759} Guruprasad and Haefliger generalized the result by Gromoll and Meyer to the setting of Riemannian orbifolds. In the present paper we generalize the result by Bangert and Franks in the following way.

\begin{thmw}\label{thm:inf_geo} On every compact Riemannian $2$-orbifold there exist infinitely many geometrically distinct closed geodesics of positive length. 
\end{thmw}

For spindle orbifolds (see Figure \ref{spindle}), which also go by the name of football orbifolds, this statement was before only known in the rotational symmetric case \cite{MR2350076}. In \cite{MR1419471} the existence of a closed geodesic in the regular part of any $2$-orbifold with isolated singularities is claimed. The alleged geodesic in the regular part is obtained as a limit of locally length minimizing curves contained in the complement of shrinking $\delta$-neighborhoods of the singular points. Note, however, that there are examples in which such curves have a limit that is not contained in the regular part. For instance, on a plane with two singular points of cone angle $\pi$ which is otherwise flat, curves of minimal length that enclose both $\delta$-neighborhoods converge to a straight segment between the singular points.

To prove our result we first reduce its statement to the case of simply connected spindle orbifolds (see Section \ref{sec:pre}). Using the curve-shortening flow we are able to prove the existence of an embedded geodesic in the regular part of any simply connected spindle orbifold and to apply ideas from Bangert's proof and Frank's result in this case. Our proof relies on the observation that embedded loops in the regular part that evolve under the curve-shortening flow either stay in the regular part forever, or collapse into a singular point in finite time (see Section \ref{sec:curve_shortening}). The possibility of a limit curve being not entirely contained in the regular part is then excluded by topological arguments (see Proposition \ref{prp:simple_geodesic}).

For spindle orbifolds with $S^1$-symmetry actually more is known, namely the existence of infinitely many distinct (even modulo isometries) closed geodesics in the regular part \cite{MR2350076}. For general spindle orbifolds (and also in other cases) this is not known to be true. For some $2$-orbifolds, e.g. for a sphere with three singular points of order $2$, even the existence of a single closed geodesic in the regular part does not seem to be rigorously proven, yet. In Section \ref{sec_geo_in_reg_part} we will add some more comments on the following questions.

\begin{question} Does there always exist a closed geodesic in the regular part of a $2$-orbifold with isolated singularities? When do there exist infinitely many?
\end{question}

Note that there are examples of surfaces with more general conical singularities that do not support infinitely many closed geodesics \cite{MR2350076}.

\section{Preliminaries}\label{sec:pre}

\subsection{Orbifolds} \label{sub:Riem_orb}
Recall that a \emph{length space} is a metric space in which the distance of any two points can be realized as the infimum of the lengths of all rectifiable paths connecting these points \cite{MR1835418}. A Riemannian orbifold can be defined as follows.
\begin{dfn} An \emph{$n$-dimensional Riemannian orbifold} $\Orb$ is a length space such that for each point $x \in \Orb$, there exists a neighborhood $U$ of $x$ in $\Orb$, an $n$-dimensional Riemannian manifold $M$ and a finite group $G$ acting by isometries on $M$ such that $U$ and $M/G$ are isometric.
\end{dfn}
Behind the above definition lies the fact that an effective isometric action of a finite group on a simply connected Riemannian manifold can be recovered from the corresponding metric quotient. In the case of spheres this is proven in \cite{MR1935486}; the general case follows in a similar way (see e.g. \cite[Lem.~133]{Lange}). In particular, the underlying topological space of a Riemannian orbifold in this sense admits a smooth orbifold structure and a compatible Riemannian structure in the usual sense (cf. \cite{MR2687544,MR1744486,MR2218759}) that in turn induces the metric. For a point $x$ on a Riemannian orbifold the isotropy group of a preimage of $x$ in a Riemannian manifold chart is uniquely determined up to conjugation. Its conjugacy class in $\Or(n)$ is denoted as $G_x$ and is called the \emph{local group} of $\Orb$ at $x$. The points with trivial local group form the \emph{regular part} of $\Orb$, which is a Riemannian manifold. All other points are called \emph{singular}. We will be particularly concerned with $2$-orbifolds all of whose singular points are isolated. In this case all local groups are cyclic and we refer to their orders as the \emph{orders of the singular points}.

We are interested in (orbifold) geodesics defined in the following way.

\begin{dfn}\label{dfn:orb_geo} An (orbifold) \emph{geodesic} on a Riemannian orbifold is a continuous path that can locally be lifted to a geodesic in a Riemannian manifold chart. A closed (orbifold) geodesic is a continuous loop that is a (orbifold) geodesic on each subinterval.
\end{dfn}

In the following, by a (closed) geodesic we always mean a (closed) orbifold geodesic. A geodesic that encounters an isolated singularity at an interior point is not locally length minimizing \cite[Thm.~3]{MR2687544}. On a $2$-orbifold such a geodesic is either reflected or goes straight through the singular point depending on whether the order of the singular point is even or odd. We say that two geodesics are \emph{geometrically distinct} if their geometric trajectories differ. Given a closed geodesic $c$ the iterations $c^m(t):=c(mt)$, $m\in \N$, form a whole \emph{tower} of geometrically equivalent closed geodesics. In the following, by infinitely many closed geodesics we always mean infinitely many geometrically distinct closed geodesics of positive length.

We need the following concept.

\begin{dfn} A \emph{covering orbifold} or \emph{orbi-cover} of a Riemannian orbifold $\Orb$ is a Riemannian orbifold $\Orb'$ together with a surjective map $\varphi : \Orb' \To \Orb$ such that each point $x\in \Orb$ has a neighborhood $U$ isometric to some $M/G$ for which each connected component $U_i$ of $\varphi^{-1}(U)$ is isometric to $M/G_i$ for some subgroup $G_i<G$ such that the isometries respect the projections.
\end{dfn}
\begin{figure}
	\centering
		\includegraphics[width=0.25\textwidth]{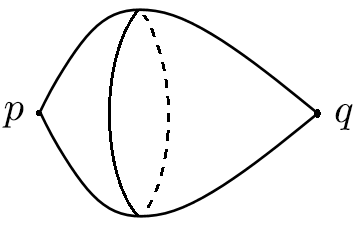}
	\caption{A \emph{$(p,q)$-spindle} orbifold $S^2(p,q)$, i.e. a $2$-orbifold with at most two isolated singularities of order $p$ and $q$ (with $p$ or $q$ perhaps being $1$). Spindle orbifolds are also known as \emph{footballs} and $(p,1)$-spindle orbifolds as \emph{teardrops}. The orbifolds in the picture are bad if and only if $p\neq q$ and simply connected as orbifolds if and only if $p$ and $q$ are coprime.}
	\label{spindle}
\end{figure}
An orbifold is called \emph{simply connected}, if it does not admit a nontrival orbi-cover. An orbifold is called \emph{good} (or developable) if it is covered by a manifold; otherwise it is called \emph{bad} \cite{Thurston}. The only bad $2$-orbifolds are depicted in Figure \ref{spindle}, cf. \cite[Thm.~2.3]{MR0705527}. In fact, every compact good $2$-orbifold is \emph{very good}, meaning that it is finitely covered by a (necessarily compact) manifold \cite[Thm.~2.5]{MR0705527}. Clearly, if an orbifold is finitely covered by an orbifold with infinitely many closed geodesics, then it has itself infinitely many closed geodesics. Since all Riemannian surfaces have infinitely many closed geodesics (see e.g. \cite[XII.5]{MR2724440} for a survey), in view of proving our main result it suffices to treat \emph{simply connected spindle orbifolds}, i.e. spindle orbifolds $S^2(p,q)$ with $p$ and $q$ coprime (see Figure \ref{spindle}).

\subsection{Orbifold loop spaces} \label{sub:orb_loop}

We would like to apply Morse theory and homological methods to find closed geodesics on orbifolds. To this end a notion of a loop space is needed. Such a notion is defined in \cite{MR2218759}. To any compact Riemannian orbifold $\Orb$ a free loop space $\Lambda \Orb$ is associated and endowed with a natural structure of a complete Riemannian Hilbert orbifold. We sketch this construction in the appendix, Section \ref{app:loop_space}.

Here we give an alternative description of $\Lambda \Orb$ in the case in which $\Orb$ has only isolated singularities. So in this section $\Orb$ will always be a \emph{Riemannian orbifold with only isolated singularities}. Let $\gamma$ be a loop on $\Orb$. In the following we always assume that such a loop $\gamma:S^1 \To \Orb$ is of class $H^1$, i.e. that it locally lifts to absolutely continuous curves on manifold charts with square-integrable velocities.
\begin{dfn} \label{dfn:develop} A \emph{development of $\gamma$} is a loop $\hat \gamma$ on a Riemannian manifold $M$ together with a map $M \To \Orb$ which is locally an orbi-covering and which projects $\hat \gamma$ to $\gamma$ (respecting the parametrizations). The development is called \emph{geodesic} if $\hat \gamma$ is a geodesic on $M$.
\end{dfn}
Every loop on $\Orb$ can be locally lifted to Riemannian manifold charts. A development of a loop $\gamma$ on $\Orb$ can be obtained by gluing together the Riemannian manifold charts that support the local lifts. In particular, this yields indeed a loop after having carried out the identifications. Two developments $(M_1,\hat \gamma_1)$ and $(M_2,\hat \gamma_2)$ are said to be equivalent if there exist neighborhoods $M_1'$ of $\hat \gamma_1$ in $M_1$ and $M_2'$ of $\hat \gamma_2$ in $M_2$ and an isometry $M_1'\To M_2'$ that maps $\hat \gamma_1$ to $\hat \gamma_2$ (respecting the parametrizations). 
\begin{dfn} \label{dfn:orbi_loop} An \emph{orbifold loop} is a loop $\gamma$ on $\Orb$ together with an equivalence class of developments of $\gamma$. The orbifold loop is called \emph{geodesic} if the developments are geodesic.
\end{dfn}
The notion of a geodesic orbifold loop is equivalent to the notion of a closed orbifold geodesic. Every geodesic orbifold loop projects to a closed orbifold geodesic in the sense of Definition \ref{dfn:orb_geo} and every closed orbifold geodesic gives rise to a unique equivalence class of geodesic developments. However, viewing a closed geodesic as a geodesic orbifold loop shows that it can be assigned local invariants like the index or the nullity as in the manifold case. Moreover, as a set the free loop space $\Lambda \Orb$ is the collection of all orbifold loops and we can even recover its metric structure using the concept of developments. Indeed, let $\mathcal{D}=(M,\hat \gamma)$ be a development defining an orbifold loop $\gamma$. The free loop space $\Lambda M$ has a natural structure of a Riemannian Hilbert manifold \cite{Klingenberg} and we have $\hat \gamma \in \Lambda M$. A loop $\hat \gamma' \in \Lambda M$ can be regarded as an orbifold loop represented by the development $(M,\hat \gamma')$. If $\gamma$ is not a constant loop at a singular point of $\Orb$, then we can choose a neighborhood $U_{\mathcal{D}}$ of $\hat \gamma$ in $\Lambda M$ such that any pair of distinct loops $\hat \gamma', \hat \gamma'' \in U_{\mathcal{D}}$ projects to distinct loops on $\Orb$ and hence corresponds to distinct orbifold loops. The $U_{\mathcal{D}}$ obtained in this way patch together to a Riemannian Hilbert manifold $\Lambda_{\mathrm{reg}} \Orb$ by identifying elements that correspond to the same orbifold loop. Indeed, if distinct $\hat \gamma_1' \in U_{\mathcal{D}_1}$ and $\hat \gamma_2' \in U_{\mathcal{D}_2}$ are identified, then, by the definition of the equivalence relation on developments, a whole open neighborhood of $\hat \gamma_1'$ in $U_{\mathcal{D}_1}$ is isometrically identified with an open neighborhood of $\hat \gamma_2'$ in $U_{\mathcal{D}_2}$. The metric completion of $\Lambda_{\mathrm{reg}} \Orb$ is the Riemannian Hilbert orbifold $\Lambda \Orb$ introduced in \cite{MR2218759}, see Section \ref{app:loop_space}, and the singular set $\Lambda \Orb \backslash \Lambda_{\mathrm{reg}} \Orb$ corresponds to the constant loops at the singular points of $\Orb$.

Given an atlas of $\Orb$ the free loop space $\Lambda \Orb$ can be written as a quotient of a Riemannian Hilbert manifold $\Omega_X$ (by a groupoid), where $X$ is the disjoint union of the manifold charts of the atlas, and this description provides local manifold charts for $\Lambda \Orb$ \cite{MR2218759}, cf. Section \ref{app:loop_space}. On $\Lambda \Orb$ the energy function $E$ is defined and its critical points correspond to the closed geodesics. Since all singular points of $\Lambda \Orb$ have zero energy an explicit knowledge of their structure will not be relevant for our argument (cf. Section \ref{appendix}). For some $\kappa>0$ we write $\Lambda^{\kappa} :=\Lambda^{\kappa} \Orb :=\Lambda \Orb \cap E^{-1}([0,\kappa))$ and for a geodesic loop $c$ with $E(c)=\kappa$ we set $\Lambda(c):=\Lambda\Orb(c):=\Lambda^{\kappa} \Orb$. The spaces $\Omega_X$ and $\Lambda \Orb$ admit finite-dimensional approximations similarly as in the manifold case, see Proposition \ref{prp:finite_approx}.

From our description of $\Lambda \Orb$ it is clear that the index $\mathrm{ind}(c)$ and the nullity $\nu(c)$ of a nontrivial orbifold geodesic can be defined as in the manifold case. Moreover, it shows that the statements used in \cite{MR0264551} on the index and the nullity of iterated geodesics and on local loop space homology remain valid in the following form since their proofs involve only local arguments. Note that there is a natural $S^1$-action on $\Lambda \Orb$ given by reparametrization.

\begin{lem}\label{lem:loop_local_structure_prop}
For a Riemannian orbifold with isolated singularities and a nontrivial orbifold geodesic $c$ on it the following statements hold true.
\begin{enumerate}
\item Either $\mathrm{ind}(c^m)=0$ for all $m$ or $\mathrm{ind}(c^m)$ grows linearly in $m$ \cite[Lem.~1]{MR0264551}.
\item There are positive integers $k_1,\ldots,k_s$ and a sequence $m_j^i \in \N$, $i>0$, $j=1,\ldots,s$, such that the numbers $m_j^i k_j$ are mutually distinct, $m_j^1=1$, $\{m_j^i k_j\}=\N$, and $\nu(c^{m_j^i k_j})=\nu(c^{k_j})$ \cite[Lem.~2]{MR0264551}.
\item There exists some $k$ such that $H_p(\Lambda(c^m)\cup S^1c^m, \Lambda(c^m))=0$ except possibly for $\mathrm{ind}(c^m)\leq p\leq\mathrm{ind}(c^m)+k$ \cite[Cor.~1]{MR0264551}.
%\item The dimensions of $H_p(\Lambda(c^m)\cup S^1c^m, \Lambda(c^m))$ are uniformly bounded independently of $p$ and $m$ \cite[Corollary~1]{MR0264551}.
\end{enumerate} 
\end{lem}

\section{Homology generated by iterated geodesics}

We will need the following slight generalizations of statements in \cite{MR0715246} by Bangert and Klingenberg. The proofs are essentially the same as in \cite{MR0715246}. For convenience we summarize the arguments in the appendix, Section \ref{appendix}. Recall that a geodesic $c$ is called \emph{homologically invisible} if $H_*(\Lambda(c)\cup S^1c,\Lambda(c))= 0$.

\begin{thm}[{cf. \cite[Thm.~3]{MR0715246}}] 
\label{thm:invisible_img} Let $\Orb$ be a compact orbifold with isolated singularities and let $c$ be a closed geodesic on $\Orb$ such that $ind(c^m)=0$ for all $m \in \N$, i.e. $c$ does not have conjugate points when defined on $\R$. Suppose $c$ is neither homologically invisible nor an absolute minimum of $E$ in its free homotopy class. Then there exist infinitely many closed geodesics on $\Orb$. 
\end{thm}
Note that if $\Orb$ is simply connected, then a nontrivial geodesic $c$ is never an absolute minimum in its free homotopy class.

\begin{lem}[{cf. \cite[Lem.~2]{MR0715246}}] \label{lem:img} Let $\Orb$ be a compact orbifold with isolated singularities and let $\{S^1c_i| i \in \N\}$ be a sequence of pairwise disjoint critical orbits such that the $c_i$ are not absolute minima of $E$ in their free homotopy classes. Suppose there exists $p \in \N$ such that $H_p(\Lambda(c_i)\cup S^1c_i, \Lambda(c_i)))\neq 0$ for all $i \in \N$. Then there exist infinitely many closed geodesics on $\Orb$. 
\end{lem}
Note that the geodesics $c_i$ in the lemma do not need to be geometrically distinct.

\section{Existence of simple closed geodesics}
\label{sec:curve_shortening}

The \emph{curve-shortening flow} can be used to prove the existence of a simple closed geodesic on any Riemannian $2$-sphere \cite{MR0979601}. In this section we discuss properties of the curve-shortening flow on Riemannian $2$-orbifolds that allow us to prove the existence of a \emph{separating geodesic} on every simply connected Riemannian spindle orbifold. Here a loop embedded in the regular part of a spindle orbifold is called \emph{separating} if each connected component of its complement contains at most one singular point. Let us first recall some well-known properties of the curve-shortening flow (cf. \cite{MR0979601,MR1888641,MR1731639,MeanCurv}). For a smoothly embedded curve $\gamma=\Gamma_0: S^1 \To M$ in a closed Riemannian surface there exists a unique maximal smooth curve-shortening flow $\Gamma_t : S^1 \To M$ for $t\in [0,T)$, $T>0$, satisfying $\frac{\partial \Gamma_t}{\partial t} = kN$ where $k$ is the curvature of $\Gamma_t$ and $N$ is its normal vector, which moreover depends continuously on the initial condition $\gamma$. This flow can be considered as the negative gradient flow of the length functional. An important feature of the curve-shortening flow is that it is a \emph{geometric flow} meaning that the evolution of the geometric image of $\gamma$ does not depend on the initial parametrization. In the situation above the curve $\Gamma_t$ is embedded for each $t\in [0,T)$. Moreover, if $T$ is finite, then $\Gamma_t$ converges to a point. If $T$ is infinite, then the curvature of $\Gamma_t$ converges to zero in the $C^{\infty}$-norm and a subsequence of $\Gamma_t$ converges to a closed embedded geodesic on $M$. In particular, $T$ is finite if the length of $\gamma$ is sufficiently small \cite[Lem.~7.1]{MR0979601}. If $M$ is not complete but the curvature is still bounded there is a single other alternative for finite $T$ namely that points on $\gamma$ do not have a limit on $M$ for $t\To T$. In particular, we see that the curve-shortening flow of an embedded curve in the regular part of a compact $2$-orbifold is (a priori) defined until the flow hits a singular point or collapses to a point. In fact, if the $2$-orbifold has only isolated singularities more is true as will be discussed below. One possibility to analyse the local behaviour of a curve-shortening flow $\Gamma_t : S^1 \To M$ on a manifold at a space-time point $(x,T)$, $ T < \infty$, is to blow up the flow at $(x,T)$ by a sequence of parabolic rescalings $((M,g),t) \To ((M,\lambda_1 g),\lambda_i^2 (t-T)+T)$, $\lambda_i \To \infty$, where $g$ denotes the Riemannian metric. Such a blow-up sequence subconverges to a self-shrinking tangent flow of an embedded curve on $T_x M$ (see \cite{MeanCurv}) which, according to a result of Abresch and Langer \cite{MR0845704}, is either a self-shrinking circle or a static straight line through the origin. In the first case $\Gamma_t$ converges to $x$ as a ``round point'' and in the second case $\Gamma_t$ is regular at $(x,T)$. Another feature of the curve-shortening flow is that it satisfies the so-called \emph{avoidance principle}, meaning that two initially disjoint curves remain disjoint under the flow. This is a consequence of the maximum principle. Moreover, by the strong maximum principle for parabolic PDE \cite[Thm.7.1.12]{MR2597943}, it is impossible for a closed curve to be disjoint from a (possible noncompact) geodesic for $t<t_0$ and to touch it tangentially at $t=t_0$.

The following statement is proven in \cite[Cor~1.7]{MR0979601} and says that the curve-shortening flow cannot spread out arbitrarily in finite time.

\begin{lem}[Grayson] \label{prp:area_rate}
If $T < \infty$, then for every $\varepsilon >0 $ there exists $t_1<T$, and an open set $U$ in $M$ such that $U$ contains every $\Gamma_t(S^1)$, $t_1<t<T$, and $U$ is contained in the $\varepsilon$-neighborhood of each $\Gamma_t(S^1)$, $t_1<t<T$.
\end{lem}

Now we analyse the evolution under the curve-shortening flow of a separating loop on a simply connected spindle orbifold.

\begin{lem}\label{prp:flow_orbifold} Let $\Orb\cong S^2(p,q)$ be a Riemannian spindle orbifold, let $\gamma$ be a separating loop on $\Orb$ and let $\Gamma: S^1 \times [0,T) \To \Orb_{\mathrm{reg}}$ be the evolution of $\gamma$ under the curve-shortening flow in the regular part of $\Orb$. If $\Gamma$ hits a singular point $x$ of order $p>2$ in finite time $T$, then the flow converges to this point.
\end{lem}
\begin{proof} By Lemma \ref{prp:area_rate} we can assume that there exists a point $y\neq x$ on $\Orb$ which is avoided by $\Gamma$ such that each $\Gamma_t$ separates $x$ and $y$. If $\Orb$ has two singular points, we choose $y$ to be a singular point. The open subset $\Orb \backslash \{y\}$ of $\Orb$ admits a $p$-fold manifold covering $M$ with a cyclic group of deck transformations $G$ of order $p$ that acts by rotations around the preimage $\hat x$ of $x$ in $M$. The preimages $\hat \Gamma_t$ of $\Gamma_t(S^1)$ in $M$ are embedded $G$-invariant loops (or loops after choosing a parametrization, which does not make a difference for us though since we are dealing with a geometric flow) and are solutions of the curve-shortening flow. By parabolically blowing up this flow at $(\hat x,T)$ as discussed above we obtain a $G$-invariant tangent flow. Because of $p=|G|>2$ this tangent flow must be a circle and so $\hat \Gamma_t$ converges to $\hat x$. For, by transversality we would otherwise obtain a contradiction to the embeddedness of the circles $\hat \Gamma_t$ for $t<T$. In particular, it follows that $\Gamma_t$ converges to $x$ in the limit $t\rightarrow T$.
\end{proof}

\begin{lem}\label{prp:flow_orbifold_gen} Let $\Orb\cong S^2(p,q)$ be a simply connected Riemannian spindle orbifold, let $\gamma$ be a separating loop on $\Orb$ and let $\Gamma: S^1 \times [0,T) \To \Orb_{\mathrm{reg}}$ be the maximal evolution of $\gamma$ under the curve-shortening flow in the regular part of $\Orb$. Then the curve $\gamma$ either
\begin{enumerate}
\item shrinks to a round point in the regular part of $\Orb$ in finite time $T$, or
\item collapses into a singular point of $\Orb$ in finite time $T$, or
\item $T=\infty$ and $\gamma$ stays in the regular part of $\Orb$ forever.
\end{enumerate}
Note that the first case can only occur if $\Orb$ has at most one singular point.
\end{lem}

\begin{proof} We only have to exclude the case that $\Gamma$ hits a singular point $x$ of $\Orb$ in finite time $T<\infty$ without collapsing into this point. Suppose this were the case. By Lemma \ref{prp:flow_orbifold} we can assume that the order $p$ of $x$ is even and that the flow does not approach a singular point of odd order at time $T$. Since $p$ and $q$ are coprime by assumption, the order $q$ must be odd. We can assume that there exists a point $y\neq x$ on $\Orb$ such that the flow avoids a whole neighborhood of $y$, such that each $\Gamma_t$ separates $x$ and $y$, and such that $\Orb \backslash \{x,y\}$ lies in the regular part of $\Orb$. In fact, for $q>1$ we can choose $y$ to be the singular point of odd order $q$, and otherwise we can apply Lemma \ref{prp:area_rate}. As in the proof of Lemma \ref{prp:flow_orbifold} the open subset $\Orb \backslash \{y\}$ of $\Orb$ admits a $p$-fold manifold covering $M$ with a cyclic group of deck transformations $G$ of order $p$ that acts by rotations around the preimage $\hat x$ of $x$ in $M$. Also, the preimages $\hat \Gamma_t$ of $\Gamma_t(S^1)$ in $M$ are embedded $G$-invariant circles (or loops after choosing a parametrization) and are solutions of the curve-shortening flow. Since $\Gamma$ avoids a neighborhood of $y$ and since $\hat \Gamma_t$ does not converge to a point at time $T$, the flow $\hat \Gamma_t$ can be extended to a flow $\hat \Gamma : S^1 \times [0,T') \To M$ with $T<T'$. Now the fact that $\Gamma$ hits $x$ at time $T$ implies by $G$-equivariance that the extended flow $\hat \Gamma$ develops a self-crossing at time $T$. This contradicts embeddedness and hence the claim follows.
\end{proof} 
Using compactness and the fact that the curvature converges to zero in infinite time, the following lemma can be proven as in the manifold case \cite[Sec.~7]{MR0979601}.
\begin{lem}\label{lem:sub_to_min_geo} In the situation of the preceding Proposition, case $(iii)$, the loop subconverges to a nontrivial (orbifold) geodesic.
\end{lem}
Note that the limit geodesic obtained in Lemma \ref{lem:sub_to_min_geo} is a priori not necessarily contained in the regular part of $\Orb$.

\begin{prp} \label{prp:simple_geodesic} On every simply connected Riemannian spindle orbifold $\Orb \cong S^2(p,q)$ there exists a separating geodesic. In particular, there exists a closed geodesic in the regular part of any spindle orbifold.
\end{prp}
\begin{rem} The first statement is optimal in the sense that there exist Riemannian metrics on $S^2(p,q)$ all of whose geodesics are closed but with only one embedded geodesic \cite{Lange2}. 
\end{rem}
\begin{proof}Choose distinct $x,y \in \Orb$ such that $\Orb \backslash \{x,y\}$ is contained in the regular part of $\Orb$. We smoothly foliate $\Orb \backslash \{x,y\}$ by circles separating $x$ and $y$. Then, under the curve-shortening flow small circles near $x$ flow to a point in finite time, and so do small circles near $y$ \cite[Lem.~7.1]{MR0979601}. However, the orientations of the limiting points will be opposite in both cases. Hence, by continuous dependence of the flow on the initial conditions there must be some circle $\gamma$ in the middle that does not flow to a point in finite time, but instead stays in the regular part of $\Orb$ forever by Proposition \ref{prp:flow_orbifold}. By Lemma \ref{lem:sub_to_min_geo} this circle subconverges to some nontrivial orbifold geodesic $c$.

It remains to show that $c$ is contained in the regular part. In fact, in this case $c$ is embedded and separating as a limit of embedded and separating loops. Let $\gamma_i$ be a subsequence of the curve-shortening flow $\Gamma_t$ that converges to $c$. Suppose that $c$ hits a singular point, say $x$ of order $p$. We can assume that the $\gamma_i$ and $c$ avoid a neighborhood of a point $y'$ on $\Orb$. The open subset $\Orb \backslash \{y'\}$ of $\Orb$ admits a $p$-fold orbi-cover $\hat \Orb$ with a cyclic group of deck-transformations $G$ of order $p$ that acts by rotations around the preimage $\hat x$ of $x$ in $\hat \Orb$. Let $s_i \in S^1$ be a sequence such that $\gamma_i(s_i)$ converges to some $c(s)\neq x$, $s\in S^1$. The restrictions of $\gamma_i$ to $S^1\backslash \{s_i\}$ and of $c$ to $S^1\backslash \{s\}$ can be lifted to embedded curves $\hat \gamma_i^j : S^1\backslash \{s_i\} \To \hat \Orb_{\mathrm{reg}}$, $j=1,\ldots,p$, and to geodesics $\hat c^j: S^1\backslash \{s\} \To \hat \Orb$, $j=1,\ldots,p$, that are permuted by the deck transformation group $G$. We can choose the $s_i$, $s$ and the $j$-numbering in such a way that $\hat \gamma_i^j$ converges to $\hat c^j$. Note that for fixed $i$ the $\hat \gamma_i^j$ have disjoint images in $\hat \Orb$ since $\gamma_i$ is embedded in $\Orb$. For $p>2$ this yields a contradiction since the geodesics $\hat c^j$ intersect transversally at $\hat x$ in this case. Suppose that $p=2$. In this case $\Orb$ is rotated by $\pi$ around $\hat x$ by the deck transformation group and the orbifold geodesic $c$ is reflected at the singular point $x$ on $\Orb$. The only way for $c$ to reverse its direction is by being reflected at a singular point of even order. Hence, by periodicity it has to be reflected at singular points of even order twice during a single period. Since $p$ and $q$ are coprime by assumption, this second reflection also has to occur at $x$. Moreover, this second reflection has to occur from a different direction, because otherwise there would have to be an additional encounter with the singularity of even order in between. Therefore, the $\hat c_j$ have transverse self-intersections at $\hat x$. Now, in this case the transversality argument from the case $p\geq 3$ yields a contradiction since the $\hat \gamma^j_i$ are embedded and hence the claim follows.
\end{proof}

\section{Proof of the main result}

As seen in Section \ref{sub:Riem_orb} it is sufficient to prove our main result for simply connected spindle orbifolds. Given the results from the preceding section, in this case a proof can be given similarly as in \cite{MR1209957} in the case of a $2$-sphere as we will discuss now.

\subsection{Outline of the proof}\label{sub:outline}
Let $\Orb \cong S^2(p,q)$ be a simply connected Riemannian spindle orbifold. By Lemma \ref{prp:simple_geodesic} there exists a separating geodesic $c: S^1=\R/\Z \To \Orb$. Suppose the following two conditions are satisfied.
\begin{compactenum}
\item For every geodesic $d:[0,\infty) \To \Orb$ with initial point $d(0)$ on $c(S^1)$ there exists $t>0$ with $d(t) \in c(S^1)$.
\item When we consider $c$ as defined on $\R$ there exists a pair of conjugate points of $c$.
\end{compactenum}
Note that the second statement is equivalent to the condition that for every $t_0 \in \R$ there exists $t_1>t_0$ such that $c(t_0)$ and $c(t_1)$ are conjugate points of $c$ \cite[Cor.~1.3]{MR2503983}. In this case Birkhoff's annulus map $B_c$ can be defined on the closed annulus $S^1 \times [0, \pi]$ as follows (cf. \cite{MR1209957,MR1161099} and  \cite[VI. 10]{MR0209095}). For $t \in S^1$ and $\alpha \in (0, \pi)$ consider a geodesic $\gamma$ starting at $c(t)$ in a direction that forms an angle of $\alpha$ with $\dot{c}(t)$. Condition $(i)$ guarantees the existence of some time $t_2$ at which $\gamma$ returns to an encounter with $c$ for the second time, say at $c(t')$. Then $\dot{\gamma}(t_2)$ and $c(t')$ enclose some angle $\alpha'\in (0, \pi)$ and one sets $B_c(t,\alpha)=(t',\alpha')$. In case of $\alpha=0,\pi$ a second conjugate point can be used to extend the map $B_c$ to all of $S^1 \times [0, \pi]$ in a continuous way. The map $B_c$ is isotopic to the identity and preserves a canonical area measure related to the Liouville measure on the unit tangent bundle of $S^2$ which is invariant under the geodesic flow. Moreover, the restrictions of $B_c$ to the two boundary components are inverse to each other. In this case the work of Franks \cite{MR1161099} implies that Birkhoff's annulus map has infinitely many periodic points (cf. \cite[Thm.~4.1]{MR1161099}) and these points correspond to closed geodesics on $\Orb$.

In the following we show that the existence of infinitely many geodesics can still be shown if Birkhoff's annulus map cannot be defined.

\subsection{Simple closed geodesics without conjugate points}

\begin{prp}\label{prp:separating_without_conjugate_points}
Suppose on a Riemannian spindle orbifold $\Orb \cong S^2(p,q)$ there exists a separating geodesic $c$ without conjugate points. Then there exist infinitely many closed geodesics. 
\end{prp}
\begin{proof} The proof works similarly as the proof of \cite[Thm.~1]{MR1209957}. We choose Riemannian manifold charts $X_0=\Orb_{\mathrm{reg}}$, $X_x$, $X_y$ of $\Orb$ where $\Orb \backslash \{x,y\} \subset \Orb_{\mathrm{reg}}$ and where $X_x$ and $X_y$ cover $\Orb \backslash \{y\}$ and $\Orb \backslash \{x\}$, respectively. Moreover, we choose finite-dimensional approximations (cf. Proposition \ref{prp:finite_approx}) $\Omega=\Omega_X^{\kappa}(k)$ and $P=\Lambda \Orb^{\kappa}(k)$ of $\Omega_X^{\kappa}$ and $\Lambda \Orb^{\kappa}$ containing $S^1c$. Since $c$ is separating, there exists a tubular neighborhood $V$ of $c(S^1)$ in the regular part of $\Orb$ which is homeomorphic to an annulus. For every $\varepsilon > 0$ we let $U(V, \varepsilon)$ denote the set of $\gamma \in P$ which have energy $E(\gamma)<E(c)+\varepsilon$ and whose projections to $\Orb$ lie in $V$ and are freely homotopic to $c$ in $V$. Since we are looking for infinitely many closed geodesics we may assume that the only closed geodesics freely homotopic to $c$ in $V$ are those in $S^1c$. Moreover, we can choose $V$ so small that every $\gamma \in U(V,\varepsilon)$ with $E(\gamma)<E(c)$ is disjoint from $c(S^1)$. This follows from the assumption that $c$ does not have conjugate points and the Gauß Lemma, cf. proof of \cite[Thm.~1]{MR1209957}. If we choose arbitrarily small $V$ and $\varepsilon$ the sets $U(V,\varepsilon)$ form a fundamental system of neighborhoods of $S^1c$ in $P$. Therefore, we have that either (cf. \cite[Thm.~1]{MR1209957})
\begin{enumerate}
\item $c$ is a local minimum of $E$ or
\item $c$ can be approximated by curves $\gamma \in P$ with $E(\gamma)<E(c)$ from both sides or
\item $c$ can be approximated by curves $\gamma \in P$ with $E(\gamma) < E(c)$ precisely from one side.
\end{enumerate}
In the second case it follows as in \cite[Thm.~1]{MR1209957} that $H_1(\Lambda(c) \cup S^1c,\Lambda(c)) \neq 0$ and this implies the existence of infinitely many closed geodesics by Theorem \ref{thm:invisible_img}.

In case $(i)$ or $(iii)$ let $D$ be a disk bounded by $c(S^1)$ such that $c$ cannot be approximated by closed curves in $D$ with $E(\gamma) < E(c)$ and suppose that $x \in D$ has order $p$. The disk $D$ is $p$-foldly covered by a disk $\hat D$ in $X_x$. A parametrization $\hat c$ of the boundary of $\hat D$ by arclength is a geodesic that covers $c$ $p$-times. For some sufficiently large $\kappa_m$ we can choose finite-dimensional approximations $\Omega_m$ of $\Omega^{\kappa_m}_X$ containing a homotopy in $\hat D$ from $\hat c^m$ to a point curve. As above we choose an annulus $V\supset c(S^1)$ so small that $S^1c$ are the only closed geodesics freely homotopic to $c$ in $V$. Moreover, we may assume that every $\gamma \in U(V,\varepsilon)$ which lies in the component $V_0$ of $V \backslash c(S^1)$ contained in $D$ has energy $E(\gamma) \geq E(c)$. Hence we have $E(\gamma)>E(c)$ for all $\gamma \in U(V,\varepsilon) \backslash S^1c$ which are contained in $D$, because otherwise such a $\gamma$ with $E(\gamma) = E(c)$ would be a closed geodesic freely homotopic to $c$. The analogous statement also holds for $c^m$ \cite[p.~5]{MR1209957}. In particular, the analogous statement also holds for $\hat c^m \in \Omega_m$. In this very situation min-max methods applied to homotopies in $\hat D$ from $\hat c^m$ to a point curve are used in the proof of \cite[Thm.~1]{MR1209957} to show the existence of a sequence of closed geodesics $\hat d_m$ in $\hat D$ such that $E(\hat d_m)$ tends to infinity and such that the local groups $H_1(\Omega_X(\hat d_m) \cup S^1c,\Omega_X(\hat d_m))$ do not vanish. The gradient of the energy functional restricted to the finite-dimensional approximation is used to deform the homotopies and so the fact that $\hat D$ is bounded by a geodesic guarantees that the construction remains in $\hat D$. The resulting geodesics project to (orbifold) geodesics $d_m$ in $D$ with $E(d_m)$ tending to infinity and with $H_1(\Lambda(d_m) \cup S^1c,\Lambda(d_m))=H_1(\Omega_X(\hat d_m) \cup S^1c,\Omega_X(\hat d_m)) \neq 0$. Therefore there exist infinitely many closed geodesics on $\Orb$ by Theorem \ref{thm:invisible_img}.
\end{proof}

\subsection{The non-Birkhoff case}
In this section we study the case of a Riemannian spindle orbifold with a separating geodesic for which the corresponding Birkhoff map is not defined. We reduce the existence of infinitely many closed geodesics to Theorem 1. Finally we explain how this implies our main result. The following lemma is a special case of \cite[Lem.~2]{MR1209957}.

\begin{lem}\label{lem:without_conj}
Suppose $c$ is a separating geodesic with conjugate points on a Riemannian spindle orbifold. Then $c$ can be approximated from either side by closed curves $\gamma$ which are disjoint from $c$ and satisfy $E(\gamma)<E(c)$. In particular, $c$ can be approximated from either side by shorter, disjoint curves.
\end{lem}

Now we can prove

\begin{prp} \label{prp:separating_with_conjugate_points}
Suppose $c$ is a separating geodesic with conjugate points on a simply connected Riemannian spindle orbifold $\Orb\cong S^2(p,q)$ and $d: (0,\infty) \To \Orb$ is a geodesic disjoint from $c$. Then there exist infinitely many closed geodesics.
\end{prp}
\begin{proof} The proof is similar as the proof of \cite[Thm.~2]{MR1209957}. However, we use the curve-shortening flow instead of Birkhoff's curve shortening process and simplify the second part of the argument.

Let $D$ be the component of $\Orb \backslash c(S^1)$ that contains $d(\R)$. Since $c$ is separating by assumption, there exists an open neighborhood $V$ of the closure of $D$ in $\Orb$ that admits a Riemannian manifold chart $\hat V$. The geodesic $d$ lifts to a geodesic $\hat d$ on $\hat V$ and the geodesic $c$ is covered $p$-times by a geodesic $\hat c$ disjoint from $\hat d$. In \cite[Thm.~2]{MR1209957} it is proven that the closure of every limit geodesic $\bar d$ of $\hat d$, that is every geodesic of the form $\bar d : \R \To \hat V$, $\bar d (t)= \exp_p(tv)$ where $(p,v)$ is an accumulation point of $(\hat d, \hat d')$ in $T \hat V$, is disjoint from $\hat c(S^1)$. In particular, the closure of the image $\tilde d$ in $V$ of such a limit geodesic of $\hat d$ is disjoint from  $c(S^1)$. Let $U_1$ be the component of $D\backslash \mathrm{closure}(\tilde d(\R))$ that contains $c(S^1)$ in its closure. Since the closure of $\tilde d(\R)$ is disjoint from $c(S^1)$, by Lemma \ref{lem:without_conj} there is an embedded loop $\gamma_1$ in $U_1$ that is freely homotopic in the regular part of $U_1$ to $c$ and shorter than $c$. We claim that the evolution $\Gamma_t$ of $\gamma_1$ under the curve-shortening flow does not leave $U_1$. Otherwise there would exist some $t_1$ minimal with the property that $\Gamma_{t_1}(S^1)$ is not contained in $U_1$. Let $x \in \Gamma_{t_1}(S^1) \cap U_1^C$. By the avoidance principle applied to $c$ and $\gamma_1$ the only possibility could be that $x$ is contained in the closure of $\tilde d(\R)$. Since $\gamma_1$ is non-contractible in $U_1$ and since $\tilde d(\R)$ is not a point, the point $x$ is regular by Lemma \ref{prp:flow_orbifold} and so is the loop $\Gamma_{t_1}$ by Lemma \ref{prp:flow_orbifold_gen}. Let $d_0:\R \To \Orb$ be the geodesic which is tangent to $\Gamma_{t_1}$ at $x$. By minimality of $t_1$ the geodesic $d_0$ is contained in the closure of $\tilde d(\R)$. In particular, the flow of $\gamma_1$ and the (static) flow of $d_0$ touch at $(x,t_1)$ for the first time. This is impossible by the maximum principle (cf. Section \ref{sec:curve_shortening}) and hence the evolution of $\gamma_1$ stays in $U_1$ as claimed. Moreover, since $\gamma_1$ is non-contractible in $U_1$ by assumption, it evolves in the regular part forever, i.e. we are in case $(iii)$ of Lemma \ref{prp:flow_orbifold}. By Lemma \ref{lem:sub_to_min_geo} the flow subconverges to a simple closed geodesic $\tilde d_1$ contained in the closure of $U_1$. The proof of Proposition \ref{prp:simple_geodesic} shows that this limit geodesic actually lies in the regular part of $\Orb$. It is distinct from $c$ since $\gamma_1$ is shorter than $c$ and the curve-shortening flow does not increase the arclength. By the choice of $\gamma_1$ the geodesic $\tilde d_1$ is separating and so we are done in this case by Proposition \ref{prp:separating_without_conjugate_points} if $\tilde d_1$ does not have conjugate points. Otherwise we define $U_2$ to be the component of $D \backslash \tilde d_1(S^1)$ whose closure contains $c$. This component is bounded by two geodesics with conjugate points and contained in the regular part of $\Orb$ by the construction of $\tilde d_1$. Moreover, it contains a non-contractible embedded loop $\gamma_2$ which is shorter than $\tilde d_1$ by Lemma \ref{lem:without_conj} and hence also shorter than $c$. Again, letting $\gamma_2$ evolve under the curve-shortening flow, the same argument as above yields a separating limit geodesic $\tilde d_2$ in $\overline{U}_2$ which is now distinct from both $c$ and $\tilde d_1$ and hence contained in $U_2$. This process can be iterated. It either yields infinitely many (simple) closed geodesics on $\Orb$ with conjugate points or terminates at a separating geodesic without conjugate points which in turn implies the existence of infinitely many closed geodesics by Proposition \ref{prp:separating_without_conjugate_points}. 
\end{proof}

From what has been said in Section \ref{sub:outline} we see that Propositions \ref{prp:separating_with_conjugate_points} and \ref{prp:separating_without_conjugate_points} imply the following statement.

\begin{prp}
Let $c$ be a separating geodesic on a Riemannian spindle orbifold $\Orb \cong S^2(p,q)$ for which Birkhoff's annulus map $B_c$ is not defined. Then there exist infinitely many closed geodesics.
\end{prp}

Recall from Section \ref{sub:outline} that Frank's work implies the existence of infinitely many closed geodesics in the case in which Birkhoff's annulus map $B_c$ can be defined (cf. Section \ref{sub:outline}). Therefore, in any case there are infinitely many closed geodesics on a Riemannian spindle orbifold. By the remark at the end of Section \ref{sub:Riem_orb} our main result follows.

\section{Closed geodesics in the regular part}\label{sec_geo_in_reg_part}

In this section we sketch some ideas and make some speculations on the question, posed in the introduction, of when there exists one, or even infinitely many closed geodesics in the regular part of a Riemannian $2$-orbifold $\Orb$ with isolated singularities. 

Let us begin with a general remark. In Section \ref{sec:curve_shortening} we have shown that an embedded loop in the regular part of a simply connected spindle orbifold cannot flow into a singular point in finite time under the curve-shortening flow unless it collapses into this point entirely. The simply-connectedness assumption was used to handle the case of singular points of order $2$ (see proof of Lemma \ref{prp:flow_orbifold_gen}). Using a local non-collasping result of Brian White this property can actually be shown to be true for embedded loops in the regular part of any Riemannian $2$-orbifold. More precisely, if such a loop were to flow into a singular point in finite time, then we could locally lift the flow to a manifold chart and look at a tangent flow of a blow-up limit at the singular space-time point in question as in the proof of Lemma \ref{prp:flow_orbifold}. Recall from that argument that such a tangent flow can only either be a self-shrinking circle or a static line, and that we used equivariance with respect to the deck transformation group in case of a singular point of order at least $3$ in order to exclude the latter case. In case of a singular point of order $2$ the blow-up sequence could in principle converge to a line, but this line would have to have multiplicity $2$, i.e. two strands of the lifted flow that are permuted by the deck transformation group would converge to it in the blow-up limit. However, this possibility is ruled out by the non-collapsing result of Brian White \cite[Thm.9.1]{MR1758759} (cf. also \cite[Sec.~7]{Chodosh} for more details).

The discussed argument not only works for embedded loops, but also for loops that stay $\delta$-embedded in the regular part under the curve-shortening flow for some $\delta>0$ as long as it remains in the regular part. Here a loop is called \emph{$\delta$-embedded} if its restriction to each subinterval of length $\delta$ with respect to arclength parametrization is embedded. If the $2$-orbifold $\Orb$ has at least $4$ singular points, then one can find infinitely many loops in the regular part with this property that are pairwise homotopically different in the regular part. This is because the property of having a ``minimal number of tranverse self-intersections'' in ones homotopy class is preserved under the curve-shortening flow. Each such loop subconverges to a closed (orbifold) geodesic on $\Orb$ and one is left to decide whether the limits lie in the regular part. If all singular points have orders at least $3$, this is the case by the same argument as in the proof of Lemma \ref{prp:simple_geodesic}, and so there exist infinitely many distinct closed geodesics in the regular part in this case. We believe that the same conclusion can be drawn in the presence of singular points of order $2$ by arguments similar as the one in the proof of Lemma \ref{prp:simple_geodesic}. For instance, if a limit geodesic hit a singular point of order $2$, then, by the above arguments, it would have to oscillate between two singular points of order $2$ and avoid all other singularities. By choosing the initial loops in a clever way, this limit behaviour could possibly be ruled out in advance.

In the case that $\Orb$ has $3$ singular points one should still be able to infer the existence of infinitely many closed geodesics if the orders of the singular points are sufficiently large, by observing that the \emph{avoidance of singularities principle} discussed above still holds for loops that ``wind around a singular point of order $2n$ up to $n-1$ times''. Otherwise, one might be able to use the existence of a finite manifold cover in this case to find, perhaps at least one, closed geodesic in the regular part.

In the case of (simply connected) spindle orbifolds the above arguments break down and we do not know how to find infinitely many closed geodesics in the regular part. There might as well exist metrics without this property.

\section{Appendix}
\label{appendix}

\subsection{Orbifold loop spaces}\label{app:loop_space} We summarize the description of an orbifold loop space from \cite{MR2218759}. For a Riemannian orbifold $\Orb$ let $X$ be the disjoint union of Riemannian manifold charts covering $\Orb$ and let $\mathcal{G}$ be the small category with set of objects $X$ and arrows being the germs of change of charts of $X$. Then $(\mathcal{G},X)$ with the usual topology of germs on $\mathcal{G}$ is an étale groupoid and $\Orb$ can be represented as a quotient $\Orb=\mathcal{G}\backslash X$ \cite[Sect.~2.1.4]{MR2218759}. A \emph{$\mathcal{G}$-loop based at $x \in X$} over a subdivision $0=t_0<t_1 \dots < t_k=1$ of the interval $[0,1]$ is a sequence $c=(g_0,c_1,g_1,\ldots,c_k,g_k)$ where
\begin{enumerate}
\item $c_i: [t_{i-1},t_i] \To X$ is of class $H^1$, i.e. $c_i$ is absolutely continuous and the velocity functions $t\mapsto |\dot{c}_i(t)|$ are square integrable.
\item $g_i$ is an element of $\mathcal{G}$ such that $\alpha(g_i)=c_{i+1}(t_i)$ for $i=0,1,\ldots,k-1$, $\omega(g_i)=c_i(t_i)$ for $i=1,\ldots,k$ and $\omega(g_0)=\alpha(g_k)=x$. Here $\alpha(g)$ and $\omega(g)$ denote the source and the target of $g\in \mathcal{G}$ (cf. \cite[Sect.~2.1.4]{MR2218759}). 
\end{enumerate}
A $\mathcal{G}$-loop is called \emph{geodesic} if all the $c_i$ are geodesics and their velocities match up at the break points $t_i$ via the $g_i$. A geodesic $\mathcal{G}$-loop gives rise to a closed geodesic on $\Orb$ in the sense of Definition \ref{dfn:orb_geo}. The space $\Omega_x$ is defined as the set of equivalence classes of $\mathcal{G}$-loops based at $x$ under the equivalence relation generated by the following operations \cite[Sect.~2.3.2, 3.3.2]{MR2218759}:
\begin{enumerate}
\item Given a $\mathcal{G}$-loop $c=(g_0,c_1,g_1,\ldots,c_k,g_k)$ over the subdivision $0=t_0<\dots<t_k=1$, we can add a subdivision point $t' \in (t_{i-1},t_i)$ together with the unit element $g'=1_{c_i(t')}$ to get a new sequence, replacing $c_i$ in $c$ by $c_i'$, $g'$,$c_i''$, where $c_i'$ and $c_i''$ are the restrictions of $c_i$ to the intervals $[t_{i-1},t']$ and $[t',t_i]$ and where $1_y$, $y\in X$ is the germ of the identity map at $y$.
\item Replace a $\mathcal{G}$-loop $c$ by a new loop $c'= (g'_0,c'_1,g'_1,\ldots,c'_k,g'_k)$ over the same subdivision as follows: for each $i=1,\ldots,k$ choose $H^1$-maps $h_i:[t_{i-1},t_i] \To \mathcal{G}$ such that $\alpha(h_i(t))=c_i(t)$, and define $c_i': t \mapsto \omega (h_i(t))$, $g_i'= h_i(t_i)g_ih_{i+1}(t_i)^{-1}$ for $i=1,\ldots,k-1$, $g_0'=g_0h_1(0)^{-1}$ and $g_k'=h_k(1)g_k$.
\end{enumerate}

The space $\Omega_X$ is defined to be $\Omega_X = \bigcup_{x\in X} \Omega_x$. It admits a natural structure of a Riemannian Hilbert manifold \cite[Sect.~3.3.2]{MR2218759}. A $\mathcal{G}$-loop $c$ gives rise to a development $(M,\hat \gamma)$ in the sense of Definition \ref{dfn:develop}. A neighborhood of the equivalence class of $c$ in $\Omega_X$ is isometric to a neighborhood of $\hat \gamma$ in $\Lambda M$. If $\Orb$ is compact, then $\Omega_X$ is complete with respect to the induced metric (cf. proof of \cite[Sect.~1.4]{Klingenberg}). An energy function can be defined on $\Omega_X$ whose critical points correspond to geodesic $\mathcal{G}$-loops \cite[Sect.~3.4.1]{MR2218759}. The groupoid $\mathcal{G}$ acts isometrically on the left on $\Omega_X$. If $[c]\in \Omega_X$ is represented by the $\mathcal{G}$-loop $c=(g_0,c_1,g_1,\ldots,c_k,g_k)$ based at $x$ and $g$ is an element of $\mathcal{G}$ with $\alpha(g)=x$ and $\omega(g)=y$, then $g[c]$ is represented by $gc:= (gg_0,c_1,g_1,\ldots,c_k,g_kg^{-1})$ \cite[Sect.~2.3.3]{MR2218759}. The quotient $|\Lambda \Orb|=\mathcal{G} \backslash \Omega_X$ is called the \emph{free loop space} of $\Orb$. The quotient map $\Omega_X \To \mathcal{G} \backslash \Omega_X$ induces a natural structure of a Riemannian Hilbert orbifold on $|\Lambda \Orb|$ denoted as $\Lambda \Orb$ \cite[Sect.~3.3.4]{MR2218759}. Since $\Omega_X$ is complete as a metric space, so is $\Lambda \Orb$. The space $\Lambda (\Orb_{\mathrm{reg}})$ is naturally a subset of $\Lambda \Orb$ and coincides with the ordinary loop space of $\Orb_{\mathrm{reg}}$ as a Riemannian manifold (cf. \cite{Klingenberg}). If $\Orb$ has only isolated singularities, then the only elements of $\Omega_X$ with nontrivial $\mathcal{G}$-isotropy are the loops that project to a singular point of $\Orb$. In this case $\mathcal{G}$-equivalence classes of $\mathcal{G}$-loops correspond to equivalence classes of developments discussed in Section \ref{sub:orb_loop} and the regular part of $\Lambda \Orb$ is isometric to the Riemannian Hilbert manifold $\Lambda_{\mathrm{reg}} \Orb$ constructed in Section \ref{sub:orb_loop}. In particular, $\Lambda \Orb$ is the metric completion of $\Lambda_{\mathrm{reg}} \Orb$, as $\Lambda_{\mathrm{reg}} \Orb$ is dense in the complete metric space $\Lambda \Orb$.

As in the manifold case the spaces $\Omega_X^{\kappa}=\Omega_X \cap E^{-1}([0,\kappa))$ and $\Lambda \Orb^{\kappa}= \Lambda \Orb^{\kappa}\cap E^{-1}([0,\kappa))$ admit finite-dimensional approximations \cite[v.1]{MR2218759}.

\begin{prp}[Finite-dimensional approximation]\label{prp:finite_approx} Let $\Orb$ be a compact orbifold and let $\kappa \geq 0$ be given. 
\begin{enumerate}
\item There exist $\epsilon >0$ and a big enough $k$ such that every element of $\Omega_X^{\kappa}$ can be represented by a $\mathcal{G}$-path $c=(g_0,c_1,g_1,\ldots,c_k,g_k)$ defined over the subdivision $0=t_0<t_1<\cdots < t_k=1$, where $t_i=i/k$, and each $c_i(t_{i-1})$ is the center of a convex open geodesic ball of radius $\epsilon$ containing the image of $c_i$.
\item The space $\Omega_X^{\kappa}$ retracts by energy-nonincreasing deformation onto the subspace $\Omega_X^{\kappa}(k)$ whose elements are represented by $\mathcal{G}$-loops as above for which each $c_i$ is a geodesic segment.
\item The restrictions of the energy function to $\Omega_X^{\kappa}$ and to $\Omega_X^{\kappa}(k)$ have the same critical points, and at each such point the nullity and the index (which can still be defined in this case \cite{MR2218759}) are the same.
\item The orbifold $\Lambda \Orb^{\kappa}:=\mathcal{G}\backslash \Omega^{\kappa}_X$ retracts by energy-nonincreasing deformation onto the finite-dimensional orbifold $\Lambda \Orb^{\kappa}(k)=\mathcal{G}\backslash \Omega_X^{\kappa}(k)$.
\end{enumerate}
\end{prp}

For a Riemannian spindle orbifold $\Orb = S^2(p,q)$ with a separating geodesic $c$ we can choose the following convenient charts. Let $x,y \in \Orb$ be two points in different components of the complement of $c$ such that $\Orb \backslash \{x,y\}$ lies in the regular part of $\Orb$. We choose $X$ to be the disjoint union of $X_0 = \Orb_{\mathrm{reg}}$ and Riemannian manifold charts $X_x$ and $X_y$ covering $\Orb\backslash \{y\}$ $p$-foldly and $\Orb\backslash \{x\}$ $q$-foldly, respectively. With respect to this choice any $\mathcal{G}$-loop based at $z \in X_0$ that projects to $\Orb_{\mathrm{reg}}$ can be represented as $(1_z,c,1_z)$ for an $H^1$-loop in $X_0$ with $c(0)=z=c(1)$ and can thus be identified with $c$. Any $\mathcal{G}$-equivalence class of $\mathcal{G}$-loops based at $z \in X_x$ which is represented by a $\mathcal{G}$-loop that projects to $\Orb\backslash \{y\}$ can be represented as $(1_z,c,g_z)$ for an $H^1$-loop in $X_0$ with $c(0)=z$ and $c(1)=gz$ and a deck transformation $g \in G_x$ of the covering $X_x \To \Orb \backslash \{y\}$.

\subsection{Homology generated by iterated geodesics}

In this section we explain why the proofs of \cite[Thm.~3]{MR0715246} and \cite[Lem.~2]{MR0715246} also work in the slightly more general situation of Theorem \ref{thm:invisible_img} and Lemma \ref{thm:loop_iteration}, which allow for isolated orbifold singularities. The first step is to obtain \cite[Thm.~1]{MR0715246} in the following version. Here the map $\psi^m:\Lambda \To \Lambda$ sends a loop $\gamma$ to its $m$-fold iteration $\gamma^m$.

\begin{thm}\label{thm:loop_iteration} Let $\Lambda$ be the free loop space of a compact Riemannian orbifold $\Orb$ with isolated singularities. For some $\kappa > 0$ let $[g]$ be a class in $\pi_p(|\Lambda|,|\Lambda^{\kappa}|)$. Then $[g^m]:=[\psi^m \circ g] \in \pi_p(|\Lambda|,|\Lambda^{\kappa m^2}|)$ is trivial for almost all $m \in \N$.
\end{thm}
In the manifold case the idea of the proof is the following. Given a homotopy $h:[a,b]\To \Lambda$, then $h^m:=\psi^m\circ h$ is a naturally associated homotopy from $h^m(a)$ to $h^m(b)$. One can replace this homotopy, which pulls the $m$ loops of $h^m(a)$ to $h^m(b)$ as a whole, by another homotopy $h_m:[a,b]\To \Lambda$ from $h_m(a)=h^m(a)$ to $h_m(b)=h^m(b)$ that pulls the $m$ loops from $h^m(a)$ to $h^m(b)$ successively. The advantage of the new homotopy over the old one is that the energy of $h_m(t)$ depends only on $h(a)$ and $h(b)$ in the limit of large $m$ and can thus be bounded appropriately. This construction can be applied fiberwise to a map $g: (I^p,\partial I^p) \To (|\Lambda|,|\Lambda^{\kappa}|)$ with respect to a splitting $I^p=I^{p-1}\times I$ and yields a homotopic map with image in $\Lambda^{\kappa}$. The same construction can be carried out in the case of an orbifold with isolated singularities. In fact, all regularity issues occurring in \cite{MR0715246} can be handled in the same way in this case since finite-dimensional approximations are also available for orbifold loop spaces (see Proposition \ref{prp:finite_approx}) and since the whole construction can be assumed to take place away from the energy zero level set, and hence in the manifold part of $\Lambda$.

The next step is to obtain a related result in homology as in \cite[Thm.~2]{MR0715246}.

\begin{thm}\label{thm:homot_to_homol} Let $\Lambda$ be the free loop space of a compact Riemannian orbifold $\Orb$ with isolated singularities. Let $H$ be an element of $H_*(\Lambda_\kappa,\Lambda^{\kappa})$, where $\kappa > 0$ and $\Lambda_\kappa$ is the union of all components of $\Lambda$ intersecting $\Lambda^\kappa$. Let $K$ be a finite set of integers $k\geq 2$. Then there exists $m\in \N$ such that no $k\in K$ divides $m$ and such that $\psi_*^m(H)$ vanishes in $H_*(\Lambda_\kappa,\Lambda^{\kappa m^2})$.
\end{thm}
The proof of \cite[Thm.~2]{MR0715246} can be taken verbatim as a proof for Theorem \ref{thm:homot_to_homol}. For sufficiently large $m$ one can construct a homotopy from a representative of $\psi_*^m(H)$ to a representative in $\Lambda^{\kappa m^2}$ by using Theorem \ref{thm:loop_iteration} inductively.

Now we explain the proof of \cite[Thm.~3]{MR0715246} and why it generalizes to the setting of Theorem \ref{thm:invisible_img}. Suppose that there exist only finitely many towers of closed geodesics on $\Orb$. Then all critical $S^1$-orbits on $\Lambda$ are isolated. Moreover, by Lemma \ref{lem:loop_local_structure_prop}, $(iii)$, and perhaps after choosing a different $c$, one can find $p\in \N$ such that $H_p(\Lambda(c)\cup S^1c,\Lambda(c))\neq 0$ and $H_q(\Lambda(d)\cup S^1d,\Lambda(d))= 0$ for every $q>p$ and every closed geodesic $d$ with $\mathrm{ind}(d^m)=0$ for all $m$. In the manifold case Lemma \ref{lem:loop_local_structure_prop}, $(ii)$, implies that there exist integers $\{k_1,\ldots,k_s\}$, $k_i\geq 2$, such that
\[
\psi_*^m:H_p(\Lambda(c) \cup S^1c,\Lambda(c)) \To H_p(\Lambda(c^m)\cup S^1c^m,\Lambda(c^m))
\]
is an isomorphism whenever none of the $k_i$ divides $m$ (for details see the proof of \cite[Thm.~2]{MR0715246}). The same conclusion holds in the present situation since its proof involves only local arguments. By Lemma \ref{lem:loop_local_structure_prop}, $(i)$, and the assumption of only finitely many towers of closed geodesics there exists some $A>0$ such that every closed geodesic $d$ with $E(d)>A$ either satisfies $\mathrm{ind}(d)>p+1$ or $\mathrm{ind}(d^m)=0$ for all $m$. Hence one has $H_{p+1}(\Lambda(d)\cup S^1d,\Lambda(d))= 0$ whenever $d$ is a closed geodesic with $E(d)>A$. Therefore, as in \cite{MR0715246}, standard arguments from Morse theory imply that $i_*:H_p(\Lambda(c^m)\cup S^1c^m,\Lambda(c^m)) \To H_p(\Lambda,\Lambda(c^m))$ is one-to-one if $E(c^m)>A$. In particular, the composition $i_* \circ \psi_*^m: H_p(\Lambda(c) \cup S^1c,\Lambda(c)) \To H_p(\Lambda,\Lambda(c^m))$ is one-to-one. Since $c$ is not an absolute minimum in its free homotopy class, this contradicts Theorem \ref{thm:homot_to_homol}.

Hence Theorem \ref{thm:invisible_img} holds. The proof of \cite[Lem.~2]{MR0715246} and Lemma \ref{thm:loop_iteration} works as follows. $H_p(\Lambda(c_i)\cup S^1c_i, \Lambda(c_i))\neq 0$ implies $\mathrm{ind}(c_i)\leq p$ by Lemma \ref{lem:loop_local_structure_prop}, $(iii)$. Since for every closed geodesic $c$ either $\mathrm{ind}(c^m)$ grows linearly with $m$ or $\mathrm{ind}(c^m)=0$ for all $m$, the $c_i$ can be iterates of a finite number of prime closed geodesics only if $\mathrm{ind}(c_i^m)=0$ for some $i$ and all $m$. In this case Theorem \ref{thm:invisible_img} proves the existence of infinitely many closed geodesics.
\newline
\newline
\textbf{Acknowledgements.} I would like to thank Nick Edelen and Alexander Lytchak for discussions and comments on a draft version of this paper. I also thank Victor Bangert and Stephan Wiesendorf for remarks and hints to the literature, respectively. I thank the referee for useful comments that helped to improve the exposition.

A partial support by the DFG funded project SFB/TRR 191 ‘Symplectic Structures in Geometry, Algebra and Dynamics' is gratefully acknowledged.

\end{document}